\documentclass[10pt, reqno]{amsart}

\usepackage{graphicx} 
\usepackage{amsmath}
\usepackage{amsfonts}
\usepackage{amssymb}

\usepackage{amssymb,latexsym,amsmath,amsfonts}
\usepackage{hyperref}
\usepackage{amsthm}
\usepackage{dsfont}
\usepackage[mathscr]{eucal}
\usepackage[english]{babel}
\usepackage{lipsum}

\theoremstyle{definition}
\newtheorem*{defn}{Definition}
\newtheorem{thm}{Theorem}[section]
\newtheorem{lem}{ \bf Lemma}[section]
\newtheorem{prop}{\bf Proposition}[section]
\newtheorem{cor}{\bf Corollary}[section]
\theoremstyle{remark}

\theoremstyle{example}

\newcommand{\be}{\begin{equation}}
\newcommand{\ee}{\end{equation}}
\newcommand{\Bea}{\begin{eqnarray*}}
\newcommand{\Eea}{\end{eqnarray*}}
\newcommand{\bea}{\begin{eqnarray}}
\newcommand{\eea}{\end{eqnarray}}

\numberwithin{equation}{section}

\begin{document}

\title{Volume growth functions of complete Riemannian manifolds with positive scalar curvature}
\author{Anushree Das \and Soma Maity }
\date{October 2024}

\address{Department of Mathematical Sciences, Indian Institute of Science Education and Research Mohali, \newline Sector 81, SAS Nagar, Punjab- 140306, India.}
\email{somamaity@iisermohali.ac.in}

\address{Department of Mathematical Sciences, Indian Institute of Science Education and Research Mohali, \newline Sector 81, SAS Nagar, Punjab- 140306, India.}
\email{ph20016@iisermohali.ac.in}


\begin{abstract} 
Let $M$ be an open manifold of dimension at least $3$, which admits a complete metric of positive scalar curvature. For a function $v$ with bounded growth of derivative, whether $M$ admits a metric of positive scalar curvature with volume growth of the same growth type as $v$ is unknown. We answer this question positively in the case of manifolds, which are infinite connected sums of closed manifolds that admit metrics of positive scalar curvature. To define a metric of positive scalar curvature with a certain volume growth type on $M$, we use the Gromov-Lawson construction of metrics with positive scalar curvature on connected sums and Grimaldi-Pansu's construction of metrics of bounded geometry of certain volume growth type on open manifolds.  We generalize this result to manifolds, which are infinite connected sums of similar closed manifolds along lower-dimensional spheres.

\end{abstract}


\maketitle

\section{Introduction} On a Riemannian manifold with positive scalar curvature, the volume of a geodesic ball of a sufficiently small radius is less than the volume of an Euclidean ball of the same radius \cite{GA}. If the scalar curvature is negative, then the reverse inequality holds. However, the positive scalar curvature is a weak condition in obtaining such a comparison of the volume of balls of large radii. Relations between volume growth and positive scalar curvature have been studied in \cite{GY}, \cite{CLS}, and \cite{MW}. In this paper, we study the functions appearing as the growth of the volume of balls on non-compact complete Riemannian manifolds with positive scalar curvature.  

A Riemannian manifold $(M,g)$ has bounded geometry if the injectivity radius $i_g\geq\iota$ and the absolute value of the sectional curvature $|K_g|\leq \kappa$ for some $\iota,\kappa>0$. Given a Riemannian manifold $(M,g)$ and a point $o\in M$, the volume growth function $v(r)$ is the volume of $B(o,r)$, the ball of radius $r$ centered at $o$. Badura, Funar, Grimaldi, and Pansu investigate those functions which are volume growth functions of a Riemannian manifold with bounded geometry, and their relations with the topology of the manifold in \cite{GP}, \cite{BM}, \cite{FG}. 

\begin{defn}
    A function $v:\mathbb{N} \to \mathbb{R}_+$ is said to have bounded growth of derivative if there exists a positive integer $L$ such that, $\forall n\in \mathbb{N}$,
    $$\frac{1}{L}\leq v(n+2)-v(n+1)\leq L(v(n+1)-v(n)).$$
\end{defn}

We call a function with bounded growth of derivative a \emph{bgd-function} in short. For any Riemannian manifold with bounded geometry, the restriction of the volume growth function to $\mathbb{N}$ is a bgd-function \cite{GP}.
Two non-decreasing functions $f,h:\mathbb{N}\rightarrow \mathbb{R}_{+}$ are said to be of the same growth type if there exists an integer $A\geq 1$ such that for all $n\in \mathbb{N}$, 
\Bea
f(n)\leq Ah(An+A)+A \ \ \text{and} \ \ h(n)\leq Af(An+A)+A. 
\Eea

The volume growth functions of a Riemannian manifold based at different basepoints all have the same growth type. In \cite{GP}, R. Grimaldi and P. Pansu classified the volume growth functions of one-ended complete Riemannian manifolds with bounded geometry as equivalence classes of bgd-functions of the same growth type. The authors generalized this to manifolds of infinitely ends and further studied the properties of these metrics in \cite{DM}. If an open manifold $M$ admits a complete metric with positive scalar curvature, then it would be interesting to find a metric of positive scalar curvature on $M$ with the volume growth function in the class of a given bgd-function.  

Let $\mathcal{M}_m$ be the set of those orientable closed manifolds of dimension $m\geq 3$ which admit a metric of positive scalar curvature. These manifolds have been studied intensively. We refer to \cite{GL}, \cite{SY}, \cite{SS}, and \cite{JR} for examples and classification of closed manifolds admitting metrics with positive scalar curvature. 
\begin{defn} \label{connectedsum} Let $\mathcal{U}$ be a collection of $m$-dimensional closed manifolds. We consider an open manifold $M$, which is a connected sum of manifolds from $\mathcal{U}$ along an infinite locally finite tree $T$, i.e. elements of $\mathcal{U}$ are placed on the vertices of the tree. An edge between two vertices represents a connected sum between the respective pieces. To take a connected sum of two $m$-dimensional closed manifolds $A$ and $B$, remove one disc of dimension $m$ from each of $A$ and $B$, and attach the resulting boundary spheres by an orientation reversing diffeomorphism. We call the resulting open manifold an infinite connected sum of elements of $\mathcal{U}$.
\end{defn}
An open Riemannian manifold is said to satisfy the uniform positive scalar curvature condition if the scalar curvature is uniformly bounded below by a positive constant $\alpha$. In \cite{BBM0}, Bessières, Besson, and Maillot showed that if an open 3-manifold admits a complete Riemannian metric of bounded geometry with uniform positive scalar curvature, then the manifold is diffeomorphic to a manifold that is an infinite connected sum along a locally finite graph of finitely many spherical manifolds. Uniformly positive scalar curvature on open 3-manifolds has also been studied in \cite{JW},\cite{BBMM}, \cite{CD}. Some topological obstructions to the existence of a complete Riemannian metric on certain open 4-manifolds with uniformly positive scalar curvature are obtained in \cite{CMM}. In this paper, we prove the following theorem.
 
\begin{thm}\label{main}
    Let $M$ be an open, orientable $m$-manifold that is an infinite connected sum of elements of $\mathcal{U}\subset\mathcal{M}_m$ for an $m\geq 3$. Given a bgd-function $v$, $M$ admits a complete Riemannian metric of bounded geometry with positive scalar curvature such that the volume growth function is in the same growth class as $v$. Moreover, if $\mathcal{U}$ is finite, then the metric has uniformly positive scalar curvature.
\end{thm}
 As per the classification in \cite{BBM0}, an open, orientable $3$-manifold that admits a metric of bounded geometry and uniform positive scalar curvature is a connected sum of finitely many elements of $\mathcal{M}_3$ along some graph. Using Theorem $2.3$ of \cite{BBM}, we can upgrade the graph to a tree at the expense of adding some $\mathbb{S}^1\times \mathbb{S}^2$ factors, which gives us the following corollary.
 \begin{cor} Let $M$ be an open, orientable $3$-manifold that admits a complete metric of bounded geometry with uniformly positive scalar curvature. Given a bgd-function $v$, $M$ admits a complete Riemannian metric of bounded geometry with uniformly positive scalar curvature such that the volume growth function is in the same growth class as $v$. 
 \end{cor}

The proof of the theorem involves constructing a metric with the desired volume growth via an adaptation of the construction used in \cite{GP} and \cite{DM}. We crucially use the construction due to M. Gromov and H. Lawson in \cite{GL} to put metrics of positive scalar curvature in connected sums. For details of this construction, we refer to \cite{MWalsh}. Alternatively, there is an equivalent construction by R. Schoen and S.T. Yau in \cite{SY} which would also let us construct a metric of positive scalar curvature on connected sums. Let $g_1$ and $g_2$ be metrics of positive scalar curvature on $X_1$ and $X_2$ respectively. The construction of the metric of positive scalar curvature on $X_1\# X_2$ by Gromov and Lawson is by removing discs $\mathbb{D}^n$ of small radii from both $X_1$ and $X_2$, where $n$ is the dimension of $X_1$ and $X_2$. A metric is then chosen on some arbitrarily small neighbourhood of each of the removed discs such that outside this neighbourhood the metric remains $g_1$ or $g_2$, respectively, while in a very small neighbourhood of the boundary spheres, it is a product metric $dt^2 + g_{S_\epsilon}$ where $g_{S_\epsilon}$ is the standard metric on a sphere of radius $\epsilon$. Here, $\epsilon$ can be chosen arbitrarily as long as it is sufficiently small. The boundary spheres are then joined to each other via an isometry to get a metric of positive scalar curvature on the entire $X_1\# X_2$.

To prove the main theorem, we start with an orientable and complete $m$-manifold, which is a connected sum of elements of $\mathcal{M}_m$ along a tree $T$. Since it is an infinite connected sum of closed manifolds, we put metrics of positive scalar curvature on each of those closed manifolds. We then modify the connected sum by also taking some connected sums with spheres, which does not change the diffeomorphism type of the manifold. The spheres are also given metrics of positive scalar curvature. The number of additional spheres taken in the connected sum is carefully chosen to ensure that the resulting metric on the entire manifold has its volume growth function in the desired growth class. The construction is inspired by Grimaldi and Pansu's proof in \cite{GP}.

The Gromov-Lawson construction of metrics with positive scalar curvature also holds in the case of connected sums of manifolds along lower-dimensional spheres. Using this, we show that the same conclusion as in Theorem \ref{main} holds for open $m$-manifolds that are infinite connected sums of elements from $\mathcal{M}_m$ along lower-dimensional spheres in Theorem \ref{main 2}. 

\subsection*{Acknowlegdement}
The first author is supported via a research grant from the National Board of Higher Mathematics, India.

\section{Construction of the metric}
Let $M$ be a complete, orientable, open manifold of dimension $m$ such that $M$ is an infinite connected sum of elements of $\mathcal{M}_m$. Represent this connected sum in the form of an infinite tree $T$, where each vertex of $T$ represents a factor of the connected sum and each edge represents the connected sum operation. Thus, each vertex of $T$ represents an element of $\mathcal{M}_m$. 

Define $Q_i$ to be the submanifold (with the boundary spheres due to the connected sum operation) of $M$ which is the union of all the components attached to the vertices on level $i$ of $T$. The number of connected components of $Q_i$ equals the number of vertices of $T$ on level $i$. The number of boundary spheres on each component of $Q_i$ is given by the degree of the vertex representing that component in $T$. We denote by $\partial^- Q_i$ the union of the boundary sphere components by which $Q_i$ is attached to the piece $Q_{i-1}$, and by $\partial^+Q_i$ we denote those boundary sphere components by which $Q_i$ is attached to $Q_{i+1}$. Note that $M$ can be recovered by attaching the components of $Q_i$ with those of $Q_{i+1}$ along the respective boundary spheres for all $i$. We want to control the volume of $M$ by adding connected sums of spheres between the $Q_i$ and $Q_{i+1}$. This does not change the diffeomorphism type of the manifold.

We start with a bgd-function $v$.  By Lemma $11$ and Lemma $10$ from \cite{GP}, we can assume without loss of generality that $v$ satisfies the following conditions:
\begin{lem}\label{lw}
Let $v : \mathbb{N} \to \mathbb{N}$ satisfy:
\begin{itemize}
\item $v(0) = 1$.
\item For all $n \in \mathbb{N}$, $2 \leq v(n + 2) - v(n + 1) \leq 2(v(n + 1) - v(n))$.
\item $v(n) = O(\lambda^n)$ for some $\lambda < 2$.
\end{itemize}
This is because, for any bgd-function $v$, it is possible to get a bgd-function of the same growth type as $v$, which satisfies the conditions of the above lemma and lets us construct an infinite tree with growth $v$. Fix a subset $S \subset N$ of vanishing lower density, i.e. $\liminf_{n\to \infty}
\frac{|S \cap \{0,...,n\}|}{n} = 0$.
There exists an admissible rooted tree ${T}_{S,v}$ with bounded geometry and with growth exactly $v$ at the root.
\end{lem}

 The process of constructing $T_{S,v}$ is as follows. The tree shall have only one end, which we call its trunk. First, choose a set $S=\{ (n_i,n_i+t_i)\}$ of vertices of the trunk which shall have only $1$ branch in our construction (i.e, these vertices shall be of degree $2$). $S$ should be chosen such that it has a vanishing lower density. The exact choice of the $S$ in our case shall be made later. We can add the additional constraint that $n_i+t_i+1\neq n_{i+1}$. Start with the root vertex. Since $v(2)=3$ by the conditions on $v$, attach $2$ branches to the root vertex, and index them. The vertex with the lowest index belongs to the trunk. Assuming we have attached the vertices up to level $i$, we now need to attach the vertices on level $i+1$. Assign an order to the $v(i)-v(i-1)$ vertices of level $i$, according to the order of the vertices on level $i-1$. That is, the vertices attached to a vertex of lower order on level $i-1$ are ordered before the vertices attached to a vertex of a higher order. This means that the vertices attached to the trunk always have the lowest order. If $i+1\in S$, attach a single vertex to the first vertex of level $i$ (the trunk); otherwise, attach $2$. For the next vertex of level $i$ in the order, attach $2$ vertices if possible without exceeding the $v(i+1)-v(i)$ vertices allowed. If not, attach $1$ vertex or none, according to the number of vertices left. Continue the process with each of the vertices on level $i$ according to their order until the $v(i+1)-v(i)$ vertices have been attached on level $i+1$. Proceeding in this manner, we construct $T_{S,v}$, which has growth exactly $v$ at the root. Note that by the construction, this tree has only a single end, which is the trunk, and all other branches terminate in finite time.

To finish the construction, attach the piece $Q_i$ collectively to the vertices on level $\{n_i,\dots,n_i+t_i\}\in S$ of the trunk. On the non-trunk vertices of $T_{S,v}$, we attach $m$-spheres. For the vertices on the trunk that do not lie in $S$, we attach one $m$-sphere, henceforth denoted as $\mathbb{S}^m$, for each component of $\partial^+Q_i$ of the $Q_i$ that precedes it. Thus, for every $i$, each component of $Q_i$ is attached to an $\mathbb{S}^m$ at each of its boundary components, and the number of $\mathbb{S}^m$ attached to a component of $Q_i$ equals the number of components of $Q_{i+1}$ that it is attached to. We attach the vertices to each other by their boundary spheres according to $T_{S,v}$, which is equivalent to taking connected sums along the edges of $T_{S,v}$. For the vertices that do not lie on the trunk but are attached to a trunk vertex via an edge, we attach one $\mathbb{S}^m$ to any one of the components of the piece representing the trunk vertex. We call the resulting manifold $R_{S,v}$. By this construction, we have attached $\mathbb{S}^m$ to the $Q_i$ along $T_{S,v}$, and each $\mathbb{S}^m$ is joined to either $1$, $2,$ or $3$ other components. Thus, each $\mathbb{S}^m$ has either $1$, $2$, or $3$ boundary components. Denote them as $S_1$, $S_2$, and $S_3$ respectively. Note that the $\mathbb{S}^m$ which are joined on the trunk are all of type $S_3$ or $S_2$, and hence the $Q_i$ are joined only to pieces of type $S_3$ or $S_2$. Again, for each sphere $S_k$ lying on level $i$ of $T$, $\partial^-{S_k}$ denotes that boundary component which joins it to a piece lying on level $i-1$, and $\partial^+ S_i$ denotes those boundary component(s), if any, which join it to piece(s) on level $i+1$. 

The admissible tree we constructed has only one end, and all other branches are finite. Hence, $R_{S,v}$ is diffeomorphic to $M$. To see this, cut the tree at the edge $\{n_j+t_j,n_j+t_j+1\}$. This results in a finite subtree, and the corresponding manifold is diffeomorphic to the submanifold of $M$ until level $j$ of $T$ (since we have only added a connected sum of finitely many spheres). Thus, we get exhaustions of $M$ and of $R_{S,v}$ such that the submanifolds are diffeomorphic at each level of the exhaustions, and the claim follows.

\section{A key Proposition }
In this section, we prove a proposition based on Proposition $13$ from \cite{GP}, that is crucially used to prove the main theorem. For a piece $P$, let $t_P$ and $T_P$ respectively denote the minimum and maximum of the distance function to $\partial^-P$, restricted to $\partial^+P$, across all components. For $k\leq T_P$, let $U_{P,k}$ denote the $k$-tubular neighbourhood of $\partial^-P$, and $v_P(k)=vol(U_{P,k})$, $v_P'(k)=v_P(k)-v_P(k-1)$. Here, a $k-$tubular neighbourhood of any set $\mathcal{A}$ in a piece $P$ refers to all points of $P$ that lie within distance $k$ of $\mathcal{A}$. When $P$ has multiple components $\{P_i\}$, note that the definition implies that $v_P(k)=\Sigma v_{P_i}(k)$. That is, we consider the volume of the tubular neighbourhoods across all the components of $P$.
\begin{prop}\label{p13}
    Let $\{Q_j$\} be a sequence of possibly disconnected compact manifolds with boundary, where $Q_j$ is the disjoint union of the pieces attached to the vertices on level $j$. Assume that
    \begin{itemize}
        \item $\partial Q_j$ is split into two collections of boundary spheres $\partial^-Q_j$ and $\partial^+ Q_j$;
        \item $\partial^- Q_{j+1}$ is diffeomorphic to $\partial^+ Q_j.$
    \end{itemize}
    Then there exist integers $l,\ h,\ H,\ d$, sequences of integers $t_j,\ U_j,$ and Riemannian metrics on pieces $Q_j,\ S_1,\ S_2,\ S_3$ such that
    \begin{enumerate}
        \item For all components $P_i$ of all pieces $P$, the maximal distance of a point of $P_i$ to $\partial^- P_i$ in that component is achieved on $\partial^+ P_i$. In other words, the maximum of those distances across the components of $P$ is equal to $T_P$.
        \item $\frac{1}{3}lt_j\leq t_{Q_j} \leq T_{Q_j}\leq lt_j$.
        \item For all other pieces $P$, $\frac{1}{3}l\leq t_{P} \leq T_{P}\leq l$.
        \item diameter$(\partial^- Q_{ji})\leq d$ on each component of $Q_{ji}$ of $Q_j$.
        \item All components $P_i$ of all pieces $P$ carry a marked point $y_{P_i}\in \partial^- P_i$. When a connected component $P_i'$ of $P'$ is glued on top of a component $P_j$ of $P$, $d(y_{P_j},y_{P_{i}'})\leq l$ (resp. $lt_j$ if $P=Q_j$).
        \item For all pieces $P$ that are not of type $Q_j$, $h\leq \textrm{min}v'_{P}\leq \textrm{max}v'_{P}\leq H$.
        \item \textrm{max} $v'_{Q_j}\leq U_j$.
        \item If $\partial^+Q_j$ and $\partial^-Q_j$ are diffeomorphic, then they are isometric, by an isometry that maps $y_j$ to $y_{j+1}$. $\partial^+ Q_j$ and $\partial^- Q_{j+1}$ are isometric for all $j$.
        \item All pieces have uniformly positive scalar curvature, bounded geometry, and product metric $g_{S_\delta}+dt^2$ near the boundary, where $g_{S_\delta}$ is the standard metric on a sphere of radius $\delta$.
         
    \end{enumerate}
    $t_j,\ d$ are respectively called the height and diameter parameters.
\end{prop}

\subsection{Metric on $S_1$, $S_2$, and $S_3$}
We denote by $g_{S_r}$ the standard metric on a $\mathbb{S}^m$ of radius $r$.
We start with the standard metric $g_{S_1}$ on the $\mathbb{S}^m$. Choosing a $\delta$ sufficiently small, remove $1$, $2$, or $3$ discs of radius $2\delta$ from $S_1$, $S_2$, or $S_3$ respectively. We apply the construction from \cite{GL}, resulting in a metric of positive scalar curvature on $P$ that agrees with $g_{S_1}$ outside small tubular neighbourhoods of the removed discs, and has the product metric $dt^2 + g_{S_\delta}$ near the boundary spheres. The tubular neighbourhoods of the boundary spheres are again chosen sufficiently small so that they are all pairwise disjoint. This gives the initial metrics of uniformly positive scalar curvature on the pieces $S_1$, $S_2$, and $S_3$. Denote the maximum of the absolute value of the sectional curvature on these metrics as $K_{max}$, and denote the minimum value of the injectivity radius as $i_{min}$. Note that $i_{min}\geq\pi\delta$.

For each piece $P$ of the type $S_1$, $S_2$, or $S_3$, $\partial^- P$ consists of one sphere, and $\partial^+ P$ is empty or consists of one or two spheres respectively. The metric near each boundary sphere is of the form $dt^2 + g_{S_\delta}$. We attach a cylinder $[0,T]\times \mathbb{S}^{m-1}$ to each of the boundary spheres of $P$ with the same product metric, choosing a suitable $T>0$. This does not change the diffeomorphism type of the piece $P$ and retains the positive scalar curvature property. Making $T$ large enough ensures that the maximum distance to any point on $\partial^- P$ is achieved on $\partial^+ P$, if it is non-empty, for each of $S_1$, $S_2$, or $S_3$, satisfying condition $1$ of Proposition \ref{p13}. Note that this does not change the bounds for the injectivity radius or the sectional curvature of the pieces.

Recall that $T_P$ is the maximum of the distances of any point in $\partial^+ P$ from $\partial^- P$, and $t_P$ is the minimum of those distances. Since any geodesic joining a point of $\partial^-P$ with a point of $\partial^+ P$ starts at one boundary sphere and ends at another, it must necessarily cross the cylinders attached to those two boundary spheres. Thus, $t_P\geq 2T$. On the other hand, the diameter of one such cylinder is bounded above by the sum of the diameter of the boundary sphere and the length $T$. Since the diameter of the boundary sphere cannot be greater than the diameter of $\mathbb{S}^m$ with the metric $g_{S_\delta}$, which we denote by $d_S$, the diameter of the cylinder is at most $d_S+T$. Thus, $T_P\leq (d_S+T)+d_S+(d_S+T)$. Choosing $T\geq \pi$ ensures that $T_P\leq 5T$. Therefore, we get $2T\leq t_P\leq T_P\leq 5T$. Define $l=6T$. Then, $$\frac{l}{3}\leq t_P\leq T_P \leq l,$$ thus satisfying condition $3$ of Proposition \ref{p13}. We choose $T$ large enough to ensure both conditions $1$ and $2$.

Set $h = \min\{v_P'\}$ and $H=\max\{v_P'\}$ amongst all $P$ of type $S_1$, $S_2$, or $S_3$. On each boundary sphere $\partial^- P$, mark an arbitrary point $y_P$. Since $T_P\leq l$, when we glue a component $P'$ on top of $P$, the gluing ensures that $d(y_P,y_{P'})\leq l$. Then the metrics defined on $S_1$, $S_2$, and $S_3$ have positive scalar curvature, bounded geometry, and satisfy the requirements stated in the proposition.

\subsection{Metric on the $Q_j$}

The piece $Q_j$ is the disjoint union of all the components lying on level $j$ of $T$. We define metrics on each component of $Q_j$ separately. Let $Q_{ji}$ be a component of $Q_j$, i.e., let $Q_{ji}$ be represented by a vertex of $T$ on level $j$. The degree of that vertex equals the number of boundary spheres on $Q_{ji}$. If the vertex has degree $k$, $\partial Q_{ji}$ consists of $k$ spheres, with $\partial^- Q$ being a single sphere and $\partial ^+ Q$ comprising of $k-1$ spheres. We first consider the closed manifold $Q_{ji}'\in \mathcal{M}_m$ represented by the vertex. There exists a metric of positive scalar curvature on $Q_{ji}'$. Choose $k$ points on $Q_{ij}'$, and choose a metric of positive scalar curvature on $Q_{ij}'$ such that the geodesic balls of radius $2\delta$ centered at those $k$ points are pairwise disjoint. This can be done by choosing any metric of positive scalar curvature and then scaling it as required. By further scaling, if required, ensure that the absolute value of the sectional curvature on $Q_{ij}'$ is bounded above by $1$, the minimum value of the injectivity radius is bounded below by $i_{min}>2\delta$. We denote this metric by $g'$. 

To get a metric on $Q_{ji}$, we start with the metric $g'$, remove the chosen points, and deform the metric in a geodesic ball of radius $2\delta$ about the removed points following the construction in \cite{GL}. Recall that for the connected sum, we remove $k$ discs from $Q_{ji}'$, and so removing the $k$ points from $Q_{ji}'$ results in a manifold diffeomorphic to $Q_{ji}$. We again get a metric of positive scalar curvature on $Q_{ji}$ such that the metric agrees with $g'$ outside small tubular neighbourhoods of the points which were removed, and on the spherical boundaries of $Q_{ji}$ it is a product metric of the form $dt^2 + g_{S_\delta}$. Furthermore, the tubular neighbourhoods of the boundary spheres on $Q_{ji}$ can be chosen such that they are all pairwise disjoint. Since the metric on the pieces of type $S_3$ defined earlier also has the product metric $dt^2 + g_{S_\delta}$ near the boundary, the boundary spheres of $Q_{ji}$ can be joined to those of $S_3$ via an isometry. We take this metric on $Q_{ji}$ and denote it by $g$. This metric has the same bounds on the sectional curvature and injectivity radius as the $S_i$ by the computation in Lemma $8.3$ of \cite{MW}, since the modification of the metric is identical to that in the case of the $S_i$ with the same initial curvature bounds. Thus, $g$ is a metric of bounded geometry on $Q_j$.

Denote one of the boundary spheres of $Q_{ji}$ as $\partial^- Q_{ji}$, and the rest as $\partial^+ Q_{ji}$. Note that this choice of boundary sphere does not change the diffeomorphism type of the resulting manifold, since this results in a change in the configuration of the discs removed for the connected sum operation, and a connected sum is independent of this choice up to diffeomorphism. To ensure that the maximum distance of any point of $Q_{ji}$ to a point of $\partial^-Q_{ji}$ is reached on $\partial^+ Q_{ji}$, we attach cylinders to the boundary spheres, which also does not change the diffeomorphism type of $Q_{ji}$. By the construction, in a neighbourhood of the boundary, the metric is a product metric of the form $dt^2 + g_{S_\delta}$. We attach a cylinder $[0,T']\times \mathbb{S}^{m-1}$ to the boundary spheres equipped with the same product metric. Choosing $T'$ large enough as before ensures that the maximum distance to a point in $\partial^- Q_{ji}$ is obtained only on $\partial^+ Q_{ji}$, and is equal to $T_{Q_{ji}}$. We denote again by $g$ the resulting metric on the piece $Q_{ji}$ along with the cylinders. This ensures that property $1$ of Proposition \ref{p13} is satisfied on the piece $Q_{ji}$. The metric on the cylinder agrees with the metric on a neighbourhood of the original boundary spheres, and hence $Q_{ji}$ still has positive scalar curvature and product metric near the boundary.

Let $d_{Q_{ji}'}$ denote the diameter of $Q_{ji}'$ with metric $g'$. $T_{Q_{ji}}$ is the maximum of the distances of any point in $\partial^+ Q_{ji}$ from $\partial^- Q_{ji}$, and $t_{Q_{ji}}$ is the minimum of the distances of a point of $\partial^+Q_{ji}$ from $\partial^-Q_{ji}$. Again, a geodesic joining a point of $\partial^-Q_{ji}$ with a point of $\partial^+ Q_{ji}$ starts at one boundary sphere and ends at another, and must thus cross the cylinders attached to those two boundary spheres. Therefore, $t_{Q_{ji}}\geq 2T'$. The diameter of such an attached cylinder is bounded above by the sum of its length $T'$ and the diameter of the boundary sphere. Since the diameter of the boundary sphere cannot be greater than the diameter of $Q_{ji}'$, which is $d_{Q_{ji}'}$, the diameter of the cylinder is at most $d_{Q_{ji}'}+T'$. Hence, $T_{Q_{ji}}\leq (d_{Q_{ji}'}+T')+d_{Q_{ji}'}+(d_{Q_{ji}'}+T')$. Choosing $T'\geq d_{Q_{ji}'}$ gives us $T_{Q_{ji}}\leq 5T'$. As a result, we again get the inequality $2T'\leq t_{Q_{ji}}\leq T_{Q_{ji}}\leq 5T'$. 

$Q_j$ is a finite disjoint union of components $Q_{ji}$ with metrics of uniformly positive scalar curvature as defined above. Consider the largest value of the length $T'$ of the attached cylinders across those components, and denote it as $T_j$. Modifying the metrics so that each component now has cylinders of length $T_j$ ensures that the inequality $2T_j\leq t_{Q_{ji}}\leq T_{Q_{ji}}\leq 5T_j$ holds for all components of $Q_j$ for the same value $T_j$. Note that increasing the length of the cylinder does not affect property $1$ of Proposition \ref{p13} or the positive scalar curvature condition. The union of all the components now gives the final metric on $Q_j$, as desired. Define $t_j=\frac{6T_j}{l}$. Then, we again have $$\frac{lt_j}{3}\leq t_{Q_j}\leq T_{Q_j}\leq lt_j$$ which establishes property $2$ of Proposition \ref{p13}.

On each component $Q_{ji}$ of $Q_j$, $\partial^- Q_{ji}$ is a sphere with metric $g_{S_\delta}$. Define $d$ as the diameter of this sphere. This ensures property $4$ of Proposition \ref{p13} holds as well. Defining $U_j$ as $\max \{v'_{Q_j}\}$ allows us to establish property $7$. Note that $\partial^+ Q_j$ and $\partial^- Q_{j+1}$ are isometric by construction. On each component $\partial^-Q_{ji}$ of $\partial^- Q_j$, mark an arbitrary point $y_{ji}$. When $\partial^+ Q_j$ and $\partial^- Q_j$ are diffeomorphic, they have the same number of boundary spheres. Then they are isometric by construction, and we can ensure that the isometry maps one marked point to another. Also, since $T_{Q_j}\leq lt_j$, whenever a piece $P$ is glued on a component $Q_{ji}$ of $Q_j$, we must necessarily have $d(y_{Q_{ji}},y_P)\leq lt_j$.

This ensures that all the properties mentioned in Proposition \ref{p13} are satisfied for the metrics defined on all the pieces, and hence finishes the proof of the proposition.

\qed

\section{Proof of Theorem \ref{main}}

 Recall that we took a connected sum along a tree $T_{S,v}$ constructed earlier in order to get a manifold $R_{S,v}$ diffeomorphic to $M$. We can define a function $r:R_{S,v}\to \mathbb{N}$ where $r$ is defined as the following. If $P = Q_j$ and $x\in P$, $r(x)=\lfloor d(x,\partial^-Q_j)\rfloor+n_jl$. If $P$ is any other type of piece, attached at a vertex of $T_{S,v}$ of level $n$, and $x \in P$, $r(x) = \lfloor d(x, \partial^-P )\rfloor + nl$. Then the discrete growth function $z$ is defined, for $n \in \mathbb{N}$, as \begin{equation}\label{e1}
    z(n)=\text{vol}\{x\in R_{S,v} |r(x)\leq n\}.
\end{equation}

The following proposition lets us choose $S$ as per our requirements. We refer to Proposition $17$ of \cite{GP} for the proof (this corresponds to the case in Proposition $17$ where $u_j$ is independent of $j$).

\begin{prop}\label{lnj}
    Let $v: \mathbb{N} \to \mathbb{N}$ be a bgd-function that satisfies the assumptions stated in Lemma \ref{lw}.
    Let $t_j$ and $d$ be the parameters of the pieces $Q_j$, as provided by the Proposition \ref{p13}. 
    Then there exists an increasing sequence $n_j$ such that
    \begin{enumerate}
        \item $n_j \geq d$;
        \item the subset $S = \bigcup_j [n_j , n_j + t_j - 1]$ has vanishing lower density;
        \item the discrete growth function $z$ of the corresponding Riemannian manifold $R_{{S,v}}$ has the same growth type as $v$.
    \end{enumerate}
\end{prop}

Using this proposition, we make a choice $S$ of vertices of the trunk, which determines a tree $T_{S,v}$ by the construction described after Lemma \ref{lw}. That in turn determines a manifold $R_{S,v}$ by taking the connected sums along $T_{S,v}$. Using Proposition \ref{p13} we get a Riemannian metric on each piece such that if a piece $P$ is attached on a piece $P'$, the boundary components on which they are attached are isometric and have product metric in some neighbourhood. We attach the pieces in this manner to define a Riemannian metric on $R_{S,v}$ such that the discrete growth function $z$ on $R_{S,v}$ lies in the same growth class as $v$. As mentioned earlier, $R_{S,v}$ is diffeomorphic to $M$ by construction. We need to show that $R_{S,v}$ with the constructed metric has a volume growth function of the same growth type as $v$.

To prove Theorem \ref{main}, it is enough to show that the volume growth function of $R_{S,v}$ lies in the same growth class as $z$, since by the previous proposition $z$ and $v$ lie in the same growth class. Choose a basepoint $o$ on $R_{S,v}$ such that $o$ lies on the piece attached to the root vertex of the tree $T_{S,v}$. Let $w(n)=$vol$(B(o,n))$ in $R_{S,v}$ with the constructed Riemannian metric. We need to show that $w$ lies in the growth class of $z$. Recall the function $r:R_{S,v}\to \mathbb{N}$ defined earlier. 

Consider a point $x$ in $R_{S,v}$ that lies on a piece $P$ on level $\alpha$. Choose a distance minimising geodesic between $o$ and $x$. Such a geodesic passes through some pieces $P_i$, where $P_i$ lies on the $i^{th}$ level of $T_{S,v}$. If the geodesic passes through some piece of the form $Q_j$, i.e if $P_i=Q_j$ for some $i$, then by construction $P_k=Q_j$ for all $n_j\leq k<n_j+t_j$. If $P_i$ consists of multiple components then the geodesic passes through exactly one component of $P_i$ by virtue of it being length minimising. Let us consider the points $y_{P_i}$ which are the marked points in $\partial^- P_i$ lying in those components of $P_i$ through which the geodesic passes. In case $P_i=Q_j$, we get only one marked point for all the $P_i$ where $n_j\leq i < n_j+t_j$. Call those selected points $y_0=o,y_1,y_2,\dots ,y_k\in \partial^- P$. Let $y_{k+1}$ be the point in $\partial^- P$ which is the closest to $o$. Consider the path generated by connecting each $y_i$ to $y_{i+1}$ via length minimising geodesics, such that $y_k$ and $y_{k+1}$ are connected via a geodesic lying entirely in $\partial^-P$. Since the distance between $o$ and $x$ is less than the length of this path, we have the following:
$$d(o,x)\leq \Sigma_{i=0}^k d(y_i,y_{i+1})+d(y_{k+1},x).$$

Based on the inequalities in condition $5$ of Proposition \ref{p13}, we know $d(y_i,y_{i+1})\leq l$ unless $y_i$ is a piece of type $Q_j$ for some $j$, in which case $d(y_i,y_{i+1})\leq lt_j$.
Hence if the piece $P$ is of the type $Q_\beta$ for some $\beta$, then $\Sigma_{i=0}^{k-1} d(y_i,y_{i+1})\leq 2n_\beta l$. Since $y_{k+1}$ lies on $\partial^- Q_\beta$, $d(y_k,y_{k+1})\leq $diameter$(\partial^- Q_\beta) = d \leq n_\beta l$. Therefore by summation, we get, 
$\Sigma_{i=0}^k d(y_i,y_{i+1})\leq 3n_\beta l$. Thus,
$$d(o,x)\leq 3n_\beta l + (r(x)-n_\beta l)\leq 3r(x).$$
If $P$ is instead of type $S_1,\  S_2$, or $S_3$, we similarly have $\Sigma_{i=0}^{k-1} d(y_i,y_{i+1})\leq 2\alpha l$, and $d(y_k,y_{k+1})\leq l$, and so 
$$d(o,x)\leq (2\alpha +1)l+(r(x)-\alpha l)\leq \frac{\alpha +1}{\alpha}r(x)\leq 3r(x).$$

Conversely, we again consider a distance minimizing geodesic $\gamma$ between $o$ and $x$, and let it pass through $\alpha$ vertices. Let $\{s_1, s_2,\dots,s_k\}$ be the values of $s$ such that $\gamma(s)$ lies on $\partial^- P'$ for some piece $P'$. Take $s_0=o$. Then $d(\gamma(s_i),\gamma(s_{i+1}))\geq \frac{l}{3}$, or  $d(\gamma(s_i),\gamma(s_{i+1}))\geq \frac{lt_j}{3}$ if the piece is of the type $Q_j$ for some $j$. And we also have $d(\gamma(s_k),x)\geq d(\partial^-(P),x)$. Therefore if $P$ is of the form $Q_\beta$ for some $\beta$,
$d(o,x)\geq \frac{1}{3}n_\beta l+d(\partial^-(P),x)\geq\frac{1}{3}n_\beta l+(r(x)-n_\beta l)\geq \frac{1}{2}r(x)$,
and in other cases, $d(o,x)\geq \frac{1}{3}\alpha l+d(\partial^-(P),x)\geq\frac{1}{3}\alpha l+(r(x)-\alpha l)\geq \frac{1}{3}r(x)$. 

Hence, combining the two directions, 
$$\{x\in R| r(x)\leq \frac{\alpha}{3}\}\leq \text{vol} \{B(o,\alpha)\}\leq \{x\in R| r(x)\leq 3\alpha\}.$$

Thus, the volume of $R_{S,v}$ is of the same growth type as the function $r$, which has the same growth type as $z$ as required, and so $R_{S,v}$ is a Riemannian manifold diffeomorphic to $M$ with volume growth in the same growth class as $v$. 

If $\mathcal{U}\subset\mathcal{M}_n$ is finite, there are finitely many compact pieces in the construction, and we have defined metrics of positive scalar curvature on each of those pieces. Hence, the infimum of the scalar curvatures on these pieces is bounded away from $0$, and hence we get a metric of uniformly positive scalar curvature. This completes the proof.\qed

A version of Gromov's construction for metrics of positive scalar curvature on connected sums also exists for positive isotropic curvature due to \cite{ML}. Hence, in the case of an open manifold $M$ of dimension $n\geq 2$ which is an infinite connected sum of closed manifolds that admit metrics of positive isotropic curvature, the same proof as above allows us to construct a metric on $M$ such that the volume function lies in the growth class of a desired bgd-function.

For dimension $3$, an orientable and complete manifold $M$ that admits a metric of uniformly positive scalar curvature is homeomorphic to an infinite connected sum of spherical manifolds and $\mathbb{S}^2\times \mathbb{S}^1$\cite{JW}. This gives us the following corollary:

\begin{cor}
    Given a $3$ dimensional orientable and complete manifold $M$ that admits a metric of uniformly positive scalar curvature and a bgd-function $v$, there exists a Riemannian manifold $(M',g)$ where $M'$ is homeomorphic to $M$ and has volume growth function in the same growth class as $v$.
\end{cor}

For $n>3$, the question of whether a general open manifold of dimension $m$ admitting a metric of uniformly positive scalar curvature can admit a metric of positive scalar curvature with a desired volume growth is still open. 

\section{Volume growth of connected sums of manifolds along lower-dimensional spheres}

\begin{defn}
    For two closed manifolds $X$ and $Y$ of dimension $n$, we can consider a connected sum of $X$ and $Y$ along spheres of codimension $q$, where $q\geq 3$. On $X$, consider an embedding of the sphere of dimension $p$ such that the embedding has a trivial normal bundle of dimension $q$, where $p+q=n$. We remove an embedded product $\mathbb{S}^{p}\times \mathbb{D}^{q}$, resulting in a boundary diffeomorphic to $\mathbb{S}^{p}\times \mathbb{S}^{q-1}$. On $Y$, we similarly consider an embedding of either $\mathbb{S}^p$ or $\mathbb{S}^{q-1}$, and remove an embedded product $\mathbb{S}^p\times \mathbb{D}^{q}$ or $\mathbb{S}^{q-1}\times \mathbb{D}^{p+1}$, again resulting in a boundary diffeomorphic to $\mathbb{S}^{p}\times \mathbb{S}^{q-1}$. Now, we use an orientation-reversing diffeomorphism to identify the two boundary components of $X$ and $Y$. The resulting manifold is called a connected sum of $X$ and $Y$ along lower-dimensional spheres.
\end{defn}

\begin{defn}
    Let $\mathcal{U}$ be a collection of $n$-dimensional closed manifolds, where $n\geq 4$. Consider an infinite rooted tree $T$. We say a manifold $M$ is an infinite connected sum along lower-dimensional spheres of manifolds from $\mathcal{U}$ along $T$ if $M$ is diffeomorphic to a manifold obtained by placing elements of $\mathcal{U}$ on the vertices of $T$, such that an edge between two vertices represents a connected sum along lower-dimensional spheres between the pieces on the respective vertices. 
\end{defn}
\begin{thm}\label{main 2}
    Let $M$ be an open, orientable $m$-manifold that is an infinite connected sum of elements of $\mathcal{U}\subset\mathcal{M}_m$ along lower dimensional spheres. Given a bgd-function $v$, $M$ admits a complete Riemannian metric of bounded geometry with positive scalar curvature such that the volume growth function is in the same growth class as $v$. Moreover, if $\mathcal{U}$ is finite, then the metric has uniformly positive scalar curvature.
\end{thm}

\begin{proof}
Consider a manifold $M$ of dimension $m$, which is a connected sum of elements of $\mathcal{M}_m$ along lower-dimensional spheres. Then $M$ is formed by taking a connected sum along some tree $T$. As before, denote by $Q_j$ the union of the elements which are attached to the vertices of level $j$, along with the boundary tori of the form $\mathbb{S}^p\times\mathbb{S}^{q-1}$ where $p+q=m$. We proceed similarly to the construction in the connected sum case. Take a bgd-function $v$ such that $v$ satisfies the conditions in Lemma \ref{lw}. Fix a subset $S$ of vanishing lower density of $\mathbb{Z}_+$, and construct an admissible tree $T_{S,v}$ with growth $v$. We again attach different pieces to the vertices of $T_{S,v}$, with the following changes. The pieces $Q_j$ now have $\mathbb{S}^p\times\mathbb{S}^{q-1}$ boundary components, and the other trunk pieces, which we call $R$, are of the form $\mathbb{S}^p\times\mathbb{S}^{q-1}\times I$ with a solid disc removed. Here $I$ is a closed interval $[0,L]$ for some $L$. The non-trunk pieces are $m$-spheres with boundary, of types $S_1$, $S_2$ and $S_3$ as before. We now construct metrics on the $Q_j$ and $R$ so that they satisfy analogous requirements to Proposition \ref{p13}. 

\subsection*{Metric on the pieces}

We choose $l>1$ and put $L=\frac{l}{3}$, so that the piece $R$ is of the form $\mathbb{S}^p\times\mathbb{S}^{q-1}\times [0,\frac{l}{3}]$. As before, we choose a $\delta>0$, where $\delta$ is small. The torus $\mathbb{S}^p\times\mathbb{S}^{q-1}$ can be endowed with the product metric $g_{S_{p,\delta}} + g_{S_{q-1,\delta}}$, where $g_{S_{k,\delta}}$ is the standard metric on a sphere of dimension $k$ of radius $\delta$. Then the product $\mathbb{S}^p\times\mathbb{S}^{q-1}\times [0,\frac{l}{3}]$ can be endowed with the metric $g_{S_{p,\delta}} + g_{S_{q-1,\delta}} + dt^2$. Now, we remove a solid disc of radius $\frac{\delta}{2}$ from the interior of $\mathbb{S}^p\times\mathbb{S}^{q-1}\times [0,\frac{l}{3}]$. We modify the metric on a small tubular neighbourhood of the resulting boundary sphere according to the construction due to Gromov-Lawson, so that in a very small neighbourhood of the boundary sphere, the metric is a product metric of the form $dt^2+g_{S_{\delta/4}}$. Note that when $\delta<\frac{l}{3}$, the condition $\frac{l}{3}\leq t_{R}\leq T_{R}\leq l$ is satisfied by the metric on $R$. We again denote the minimum injectivity radius by $i_{min}$ and the maximum absolute value of the sectional curvatures by $K_{max}$. Note that since $\delta$ is small, $K_{max}>1$.

For the metric on the $Q_j$, we follow the same construction on each component as before, except we now remove some embedded copies of $\mathbb{S}^p\times \mathbb{D}^q$ to result in boundaries of the form $\mathbb{S}^p\times \mathbb{S}^{q-1}$. Consider a metric of positive scalar curvature on $Q_j'\in \mathcal{M}_m$, and remove the embedded tori resulting in the boundary components. Scaling up the metric, if required, again ensures that the absolute values of the sectional curvatures are bounded above by $1$, and the injectivity radius is bounded below by $i_{min}$. We scale the metric further to obtain a tubular neighbourhood of radius $2\delta$ of each embedded $\mathbb{S}^p$, thus giving us an embedded $\mathbb{S}^p\times\mathbb{D}^q(2\delta)$. We identify $\mathbb{S}^p\times\mathbb{D}^q(2\delta)\setminus \{0\}$ with $\mathbb{S}^p\times\mathbb{S}^{q-1}\times (0,2\delta)$. Following the construction by Gromov and Lawson, we modify the metric on $\mathbb{S}^p\times\mathbb{S}^{q-1}\times (0,2\delta)$ so that near the boundary $\mathbb{S}^p\times\mathbb{D}^q(2\delta)\times \{0\}$ it is a product metric $dt^2+g_{S_{p,\delta_1}} + g_{S_{q-1,\delta}}$ for some $\delta_1,\delta_2>0$. Scale the metric again to ensure $\delta_1,\delta_2>\delta$. Note that the $\delta_1$ and $\delta_2$ might be different for the different boundary components, but since there are finitely many such components such scaling is possible. Using Lemma $3$ of \cite{GL}, we get an interval $[0,a]$ such that there is a positive scalar curvature metric on the cylinder $\mathbb{S}^p\times \mathbb{S}^{q-1}\times [0,a]$ which restricts to the metric $dt^2+g_{S_{p,\delta_1}} + g_{S_{q-1,\delta_2}}$ near one boundary and to $dt^2+g_{S_{p,\delta}} + g_{S_{q-1,\delta}}$ near the other boundary. We attach this cylinder to the boundary component of $Q_j$ via the isometry, which results in the new boundary component having a metric of the form $dt^2+g_{S_{p,\delta}} + g_{S_{q-1,\delta}}$ in a small enough neighbourhood. Thus, the $Q_j$ can be attached to the $R$ via their boundary tori, since the boundaries are isometric and both manifolds have product metrics near the boundary. The metric on $Q_j$ has the bounds of $K_{max}$ for the absolute values of the sectional curvatures and $i_{min}$ for the injectivity radius. This is because the initial metric has the absolute value of sectional curvature bounded above by $1$, and on deforming it, the sectional curvature increases until the torus with metric $g_{S_{p,\delta}} + g_{S_{q-1,\delta}}$, following the computation in \cite{MW}. The absolute values of the sectional curvatures of this torus are bounded above by $K_{max}$, which gives us the bound for the $Q_j$. For the pieces $S_1$ and $S_2$, the metric is constructed as in the connected sum case, and near the boundary spheres, the metric is of the form $dt^2+g_{S_{\delta/4}}$. If $\mathcal{U}\subset\mathcal{M}_n$ is finite, there are finitely many compact pieces in the construction, and we have defined metrics of positive scalar curvature on each of those pieces. Hence, the infimum of the scalar curvatures on these pieces is bounded away from $0$, and hence we get a metric of uniformly positive scalar curvature.

With the metric defined on all the pieces, attaching the pieces along $T_{S,v}$ gives rise to a metric on the resulting manifold which is diffeomorphic to $M$. Choosing the set $S$ suitably according to Proposition \ref{lnj}, a similar argument as before shows that the volume growth function of the manifold is of the same growth type as $v$. This completes the proof.

\end{proof}


\end{document}